\documentclass{article}
\usepackage{amsmath,amssymb,amsthm,graphicx,epsfig,float,url}
\usepackage[colorlinks=true]{hyperref}
\usepackage{pdfsync}
\usepackage[usenames, dvipsnames]{xcolor}
\usepackage{tikz}
\usepackage{subfig}
\usepackage{stmaryrd}
\usepackage{microtype}
\usepackage{multirow}
\usepackage{float}
\usepackage{bbm}
\usepackage{color}

\newcommand{\R}{\textnormal{I\kern-0.21emR}}
\newcommand{\N}{\textnormal{I\kern-0.21emN}}

\renewcommand{\geq}{\geqslant}
\renewcommand{\leq}{\leqslant}

\newtheorem{theorem}{Theorem}

\theoremstyle{definition}

\title{\bf Minimal cost-time strategies for population replacement using the IIT}

\author{Antoine Henrot\footnote{Universit\'e de Lorraine, CNRS, Institut Elie Cartan de Lorraine, BP 70239 54506 Vand\oe uvre-l\`es-Nancy Cedex, France ({\tt antoine.henrot@univ-lorraine.fr}).}
	\and Idriss Mazari\footnote{Sorbonne Universit\'es, UPMC Univ Paris 06, UMR 7598, Laboratoire Jacques-Louis Lions, F-75005, Paris, France (\texttt{idriss.mazari@upmc.fr}).}
	\and Yannick Privat\footnote{IRMA, Universit\'e de Strasbourg, CNRS UMR 7501, 7 rue Ren\'e Descartes, 67084 Strasbourg, France ({\tt yannick.privat@unistra.fr}).}
}

\author{Luis Almeida\footnote{Sorbonne Universit\'e, CNRS, Universit\'e de Paris, Inria, Laboratoire J.-L. Lions, 75005 Paris, France ({\tt luis.almeida@sorbonne-universite.fr}).}
	\and Jes\'us Bellver Arnau\footnote{Sorbonne Universit\'e, CNRS, Universit\'e de Paris, Inria, Laboratoire J.-L. Lions, 75005 Paris, France ({\tt bellver@ljll.math.upmc.fr}).}
	\and Michel Duprez\footnote{CEREMADE, Université Paris-Dauphine \& CNRS UMR  7534,Université PSL, 75016 Paris, France ({\tt duprez@ceremade.dauphine.fr}).}
	\and Yannick Privat\footnote{IRMA, Universit\'e de Strasbourg, CNRS UMR 7501, Inria, 7 rue Ren\'e Descartes, 67084 Strasbourg, France ({\tt yannick.privat@unistra.fr}).}
}

\date{}

\begin{document}

\maketitle

\begin{abstract}
Vector control plays a central role in the fight against vector-borne diseases and, in particular, arboviruses. The use of the endosymbiotic bacterium \textit{Wolbachia} has proven effective in preventing the transmission of some of these viruses between mosquitoes and humans, making it a promising control tool. The Incompatible Insect Technique (IIT) consists in replacing the wild population by a population carrying the aforementioned bacterium, thereby preventing outbreaks of the associated vector-borne diseases.
In this work, we consider a two species model incorporating both Wolbachia infected and wild mosquitoes. Our system can be controlled thanks to a term representing an artificial introduction of Wolbachia-infected mosquitoes. Under the assumption that the birth rate of mosquitoes is high, we may reduce the model to a simpler one on the proportion of infected mosquitoes. We investigate minimal cost-time strategies to achieve a population replacement both analytically and numerically for the simplified 1D model and only numerically for the full 2D system
\end{abstract}


\medskip

\section{Introduction} 

Arboviruses are a major threat for human health throughout the world, being responsible for diseases such as Dengue, Zika, Chikungunya or Yellow fever \cite{dengue1, dengue2}. This has led to the development of increasingly sophisticated techniques to fight against these viruses, specially, techniques targeting the vector transmitting the diseases, i.e. the mosquito \cite{bliman_sit, anguelov_dumont, alphey_10, alphey_14}.
 Recently, there has been increasing interest in using the well-known bacterium Wolbachia which lives only inside insect cells \cite{wolbach}, as a tool for carrying out this vector-targeted control  \cite{bliman_iit, hughes_briton, schraiber, ndii_hickson, hoffman_montgomery, sallet}. 
This method is known as the ``Incompatible Insect Technique'' (IIT). Mosquitoes carrying this bacterium show a significant reduction in their vectorial capacity \cite{walker,moreira_iturbe, turley_moreira, mousson}. The fact that the bacterium is transmitted from the mother to offspring is a phenomenon called cytoplasmic incompatibility (CI) \cite{sinkins, CI}, which produces cross sterility between infected males and uninfected females. It is a promising tool for fighting against these diseases by performing a replacement of the mosquito population.

In this work, we focus on studying the minimal amount of mosquitoes needed to ensure an effective population replacement. We will first fix the time horizon over which we will be able to act and consider a cost functional modeling the population replacement strategy. Then, a convex combination of the time horizon and the previous cost functional will be considered to get an optimized cost-time strategy. To investigate this issue, let us introduce a model for two interacting mosquito populations: a Wolbachia-free population $n_1$, and a Wolbachia carrying one, $n_2$.
The resulting system, introduced in \cite{wolbachia}, reads
\begin{equation}\label{eq:fullsys}
\begin{cases}
& \frac{dn_1(t)}{dt}=b_1n_1(t)\left(1-s_h\frac{n_2(t)}{n_1(t)+n_2(t)}\right)\left(1-\frac{n_1(t)+n_2(t)}{K}\right)-d_1n_1(t),\\
& \frac{dn_2(t)}{dt}=b_2n_2(t)\left(1-\frac{n_1(t)+n_2(t)}{K}\right)-d_2n_2(t)+u(t) \mbox{ , } t>0,\\
& n_1(0)=n_1^0 \mbox{ , } n_2(0)=n_2^0,
\end{cases}
\end{equation}
where $u(\cdot)$ plays the role of a control function that will be made precise in the sequel.

The aim is to achieve a population replacement meaning that, starting from the equilibrium $\left(n_1^*,0\right)=\left(K\left(1-\frac{d_1}{b_1}\right),0\right)$, one wants to reach the equilibrium $\left(0,n_2^*\right)=\left(0, K\left(1-\frac{d_2}{b_2}\right)\right)$. In the system, the positive constant $K$ represents the carrying capacity of the environment, $d_i$ and $b_i$, $i=1,2$ denote respectively the death and birth rates of the mosquitoes and $s_h$ is the CI rate.

To act on the system, we will use the control function appearing in the second equation, representing the rate at which Wolbachia-infected mosquitoes are introduced into the population. We will impose the following biological constraint on the control function $u$: the rate at which we can instantaneously release the mosquitoes will be bounded above by $M$. Another natural biological constraint is to limit the total amount of mosquitoes we release up to time $T$, as done in \cite{control} and \cite{wolbachia}. In this article, we take a different approach, looking actually at minimizing the total number of mosquitoes used. We thus introduce the space of admissible controls 
\begin{equation}\label{def:UTM}
\mathcal{U}_{T,M}:=\left\{u\in L^{\infty}\left(\left[0,T\right]\right),0\leq u\leq M\mbox{ a.e. in }(0,T)\right\}.
\end{equation}

Before going further, in order to simplify the analytical study of the system, we will work on a reduction of the problem already used in \cite{wolbachia}. It is shown there that, under the hypothesis of having a high birth rate, i.e. considering $b_1=\frac{b_1^0}{\varepsilon}$, $b_2=\frac{b_2^0}{\varepsilon}$ and taking the limit $\varepsilon\to 0$, the proportion $\frac{n_2}{n_1+n_2}$ of Wolbachia-infected mosquitoes in the population uniformly converges to $p$, the solution of a simple scalar ODE, namely 
\begin{equation}\label{def:p1D}
\left\{\begin{array}{ll}
\frac{dp}{dt}(t)=f(p(t))+u(t)g(p(t)),& t>0\\
p(0)=0, & 
\end{array}\right.
\end{equation}
where $$f(p)=p(1-p)\frac{d_1b_2^0-d_2b_1^0(1-s_hp)}{b_1^0(1-p)(1-s_hp)+b_2^0p} \mbox{ and } g(p)=\frac{1}{K}\frac{b_1^0(1-p)(1-s_hp)}{b_1^0(1-p)(1-s_hp)+b_2^0p}.$$

We remark that in the absence of a control function, the mosquito proportion equation simplifies into $\frac{dp}{dt}=f(p)$. This system is bistable, with two stable equilibria at $p=0$ and $p=1$ and one unstable equilibrium at  $p=\theta=\frac{1}{s_h}\left(1-\frac{d_1 b_2^0}{d_2 b_1^0}\right)$, the only root of $f$ strictly between $0$ and $1$ assuming that 
\begin{equation}\label{apero1908}
1-s_h<\frac{d_1 b_2^0}{d_2 b_1^0}<1.
\end{equation}
In what follows, it will be useful to notice that the derivative of the function $f/g$ has a unique zero $p^*$ in $(0,\theta)$ defined by 
\begin{equation}\label{def:petoile}
p^*=\frac{1}{s_h}\left(1-\sqrt{\frac{d_1b_2^0}{d_2b_10}}\right).
\end{equation}
In \cite{wolbachia}, for the aforementioned system of two equations, the problem 
\begin{equation}\label{full:CP}
\inf\limits_{u\in\mathcal{U}_{T,C,M}}J(u), \quad\text{with }J(u)=\frac{1}{2}n_1(T)^2+\frac{1}{2}\left[(n_2^*-n_2(T))_+\right]^2
\end{equation}
and $\mathcal{U}_{T,C,M}=\left\{u\in\mathcal{U}_{T,M} , \int_0^Tu(t)dt\leq C\right\}$ is considered. Then, denoting by $J^{\varepsilon}(u)$ the criterion $J(u)$ where the birth rates $b_1$ and $b_2$ have been respectively replaced by $\frac{b_1^0}{\varepsilon}$ and $\frac{b_2^0}{\varepsilon}$, with $\varepsilon>0$, a $\Gamma$-convergence type result is proven. More precisely,  any solution of the reduced problem is close to the solutions of the original problem for the weak-star topology of $L^{\infty}(0,T)$ and moreover
 $$
 \lim_{\varepsilon\to 0}\inf_{u\in\mathcal{U}_{T,C,M}}J^{\varepsilon}(u)=\inf_{u\in\mathcal{U}_{T,C,M}}J^{0}(u) ,
 $$
where
\begin{equation}\label{def:J0}
J^{0}(u)=\lim\limits_{\varepsilon\to 0} J^{\varepsilon}(u)=K(1-p(T))^2
\end{equation}
and $p$ is the solution of \eqref{def:p1D} associated to the control function choice $u(\cdot)$.
The arguments exposed in \cite{wolbachia} can be adapted without effort to our problem, allowing us to investigate the minimization of $J^0$ given by \eqref{def:J0}, which is easier to study both analytically and numerically,  instead of the full problem  \eqref{full:CP}, since the solutions of both problems will be close in the sense of the $\Gamma$-convergence.

In accordance with the stability considerations above concerning system \eqref{def:p1D} without control, we will impose as final state constraint $p(T)=\theta$ since once we are above this state, the system will evolve by itself (with no need to control it any longer) to $p=1$, the state of total invasion. 
By analogy with the two-equation system, $\theta$ represents the threshold of the basin of attraction of the equilibrium $(0,n^*_2)$.
Since our goal in this work is to minimize the cost of our action on the system, we will address the issue of minimizing the total number mosquitoe releaseds, in other words, the integral over time of the rate at which mosquitoes are being released, namely 
\begin{equation}\label{def:J}
J(u)=\int_{0}^{T}u(s)\, ds.
\end{equation}
\section{Optimal control with a finite horizon of time} \label{sec:Tfix}

We first consider an optimal control problem where the time window $[0,T]$ in which we are going to act on the system is fixed. This leads us to deal with the optimal control problem
\begin{equation}
\label{prob:Tfixed}
\tag{$\mathcal{P}_{T,M}$}
 \begin{cases}
  \inf\limits_{u\in\mathcal{U}_{T,M}} J(u) , \\
 p'=f(p)+ug(p), \ p(0)=0 \ , \ p(T)=\theta \ ,
 \end{cases}    
\end{equation}
where $J(u)$ is defined by \eqref{def:J} and $\mathcal{U}_{T,M}$ is given by \eqref{def:UTM}.

\begin{theorem}
\label{theo:Tfixed}
Let us introduce
\begin{equation}\label{def:Tstarmstar}
m^*:=\max\limits_{p\in[0,\theta]}\left(-\frac{f(p)}{g(p)}\right)\quad \text{and}\quad 
T^*=\int_0^{\theta}\frac{d\nu}{f(\nu)+Mg(\nu)}
\mbox{ for }M>0.
\end{equation}
Let us assume that $M>m^*$, $T\geq T^*$ and that \eqref{apero1908} is true. Then there exists $u^*\in\mathcal{U}_{T,M}$ bang-bang solving problem \eqref{prob:Tfixed}. Furthermore, every function $u^*$ defined by $u^*_{\xi}=M\mathbbm{1}_{(0+\xi,T^*+\xi)}$, with $\xi\in[0,T-T^*]$ solves the problem, and $J(u_{\xi}^*)=MT^*$.
\end{theorem}

\begin{proof}
Observe first that $T^*$ is constructed to be the exact time such that $p(T)=\theta$ whenever one takes the ``maximal'' control equal to $M$ on $(0,T^*)$. For this reason, if $T=T^*$ the set of admissible controls reduces to a singleton and we will assume from now on that $T>T^*$.

For the sake of readability, the proof of the existence of solutions is going to be treated separately in Appendix \ref{app:existT}. We focus here on deriving and exploiting the necessary optimality conditions. To this aim, let $u^*$ be a solution of Problem~\eqref{prob:Tfixed}.
To apply the Pontryagin's Maximum Principle (PMP), let us introduce ${U}=[0,M]$, as well as the Hamiltonian $\mathcal H$ of the system, given by
$$
\mathcal H:\R_+\times \R\times \R\times \{0,-1\}\times U\ni (t,p,q,q^0,u)\mapsto q(f(p)+ug(p)) + q^0 u .
$$
The equation $q$ obeys reads
$$
q'(t)=-\frac{\partial \mathcal H}{\partial p}=-q(f'(p)+ug'(p)), \quad t\in (0,T) 
$$
so that $q(t)=q(0)e^{-\int_0^t f'(p(s))+u(s)g'(p(s))ds}$, and thus $q(t)$ has a constant sign.

The instantaneous maximization condition reads
\begin{eqnarray}\label{metz16H31}
u^*(t)&\in&\arg\max\limits_{v \in U}\mathcal H(t,p,q,q^0,v)
=\arg\max\limits_{v \in U}\left(qg(p)+q^0\right)v.
\end{eqnarray}
%
%
This condition implies that $q$ is positive in $(0,T)$. Indeed, assume by contradiction that $q\leq 0$ in $(0,T)$. Since $q$ has a constant sign, since $q^0\in\{0,1\}$, and since the pair $(q,q^0)$ is nontrivial according to the Pontryagin Maximum Principle, there are two possibilities: either $q^0=-1$ and $q\leq 0$ in $(0,T)$ or $q^0=0$ and $q< 0$ in $(0,T)$. In both cases, the optimality condition yields that $u^*(t)=0$ for all $t\in[0,T]$, which is in contradiction with the condition $p(T)=\theta$. We thus get that $q(0)>0$ or similarly that $q(t)>0$ in $(0,T)$. 

Now, let us show that $q^0=-1$. To this aim, let us assume by contradiction that $q^0=0$. Hence, the optimality condition reads
$u^*(t)\in\arg\max\limits_{v \in U}qg(p)v$,
and since both $q$ and $g\circ p$ are positive in $[0,T]$, then one has $u^*(t)=M\mathbbm{1}_{[0,T]}$. But this control function verifies the final state constraint $p(T)=\theta$ if and only if $T=T^*$. 
We have thus reached a contradiction if $T>T^*$. We have therefore obtained that $q^0=-1$.

Let us introduce the function $w$ given by $w(t)=q(t)g(p(t))$. The maximization condition \eqref{metz16H31} yields
\begin{equation*}
     \begin{cases}
      w(t)\leq 1 \mbox{ on } \{u^*=0\} ,  \\
      w(t)=1 \mbox{ on } \{0<u^*<M\} , \\
      w(t)\geq 1 \mbox{ on } \{u^*=M\} .\\
     \end{cases}    
\end{equation*}
Let us finally prove that any solution is bang-bang. The following approach is mainly inspired from \cite[Proof of Lemma~7]{wolbachia}. For the sake of readability, we recall the main steps and refer to this reference for further details. 

Note that \begin{eqnarray*}
w'(t)&=&q'(t)g(p(t))+q(t)g'(p(t))p'(t)\\
&=&-q(t)\left(f'(p(t))+u(t)g'(p(t))\right)g(p(t))\\
&& +q(t)g'(p(t))\left(f(p(t))+u(t)g(p(t))\right)\\
&=&q(t)\left(-f'(p(t))g(p(t))+f(p(t))g'(p(t))\right)\\
&=&q(t)g(p(t))^2\left(-\frac{f}{g}\right)'(p(t)).
\end{eqnarray*}
Looking at the monotonicity of $p\mapsto -f(p)/g(p)$, we deduce that $w$ is increasing if $0<p(t)<p^*$ according to \eqref{apero1908}, and decreasing if $p^*<p(t)<\theta$, where $p^*$ is defined by \eqref{def:petoile}. Recall that one has $0<p^*<\theta$ according to \eqref{apero1908}.

To prove that $u^*$ is bang-bang, we show that $w$ cannot be constant on a measurable set of positive measure.
By contradiction, assume that $w$ is constant on a measurable set $I$ of positive measure. Necessarily, we also have $\left(-f/g\right)'(p(t))=0$ on $I$ which implies that $p(t)=p^*$ on $I$. This implies that for a.e. $t\in I$, $p(t)$ is constant. In order to get that, we must have, $p'=0$ on that set\footnote{Recall that if a function $F$ in $H^1(0,T)$ is constant on a measurable subset of positive Lebesgue measure, then its derivative is equal to 0, a.e. on $I$, see e.g. \cite[Lemma 3.1.8]{HenrPierre}}. At this step, one has
$$
\{0<u^*<M\}\subset \{t\in (0,T)\mid p(t)=p^*\}\subset \{t\in (0,T)\mid u^*(t)=-f(p^*)/g(p^*)\}.
$$
Since $M>m^*$, one has $-f(p^*)/g(p^*)\in (0,M)$ which shows that the converse inclusion is true, and therefore,
$$
\{0<u^*<M\}= \{t\in (0,T)\mid u^*(t)=-f(p^*)/g(p^*)\}.
$$
Using that $\{w=1\}\subset \{p=p^*\}$, one gets 
$$
\{0<u^*<M\}=\{w=1\}, \quad \{u^*=M\}=\{w>1\}, \quad \{u^*=0\}=\{w<1\}
$$
and in particular, $\{u^*=M\}$ and $\{u^*=0\}$ are open sets.

Let $I$ be a maximal interval on which $p$ is equal to $p^*$. On $I$, we have  $u^*(t) =-f(p^*)/g(p^*)$. If, at the end of the interval, we have $u = 0$ then $p$ must decrease since $p^* < \theta$, and therefore $\left(-f/g\right)'(p(t))>0$ so that $w$ increases. But this is in contradiction with the necessary
optimality conditions (on $\{u^*=0\}$, one has $w\leq 1$). If at the end of the interval we have $u=M$ then $p$ increases and $w$ decrease, leading again to a contradiction. Hence we must have $|I|=0$ or $I =[0,T]$. Since $p(t)=p^*$ on $I$, $p(0)=0$ and $p(T)=\theta$, this is impossible.

Therefore, $u^*$ is equal to $0$ or $M$ almost everywhere, meaning that it is bang-bang.

Finally, let us prove that the set $\{u^*=M\}$ is a single interval. We argue by contradiction. Using the fact that the solution is bang-bang, $M>m^*$ and that $\{u^*=M\}$ is open, if $\{u^*=M\}$ is not one open interval, there exists $(t_1,t_2)$ on which $u^*=0$ and $p'<0$. Since the final state is fixed, $p(T)=\theta$, then there necessarily exists a time $t_3>t_2$ such that $p(t_3)=p(t_1)$. Let us define $\tilde{u}$ as
$$
\tilde{u}(t)=\begin{cases}
 0 & t\in  (0,t_3-t_1), \\
 u^*(t-t_3+t_1) & t\in (t_3-t_1,t_3),\\
 u^*(t) &t\in  (t_3,T).
\end{cases}$$
We can easily check that since $u^*\in\mathcal{U}_{T,M}$, then $\tilde{u}\in\mathcal{U}_{T,M}$ and if $p$ defined as the solution to
 $$
 \left\{\begin{array}{ll}
 p'(t)=f(p(t))+u^*(t)g(p(t)), & t\in [0,T]\\
 p(0)=0 & 
 \end{array}\right.
$$
satisfies $p(T)=\theta$, then the function $\tilde p$ defined as the solution to
 $$
 \left\{\begin{array}{ll}
 \tilde p'(t)=f(\tilde p(t))+\tilde{u}(t)g(\tilde p(t)), & t\in 
[0,T]\\
 \tilde p(0)=0 & 
 \end{array}\right.
$$ 
satisfies $\tilde p(T)=\theta$ and is moreover non-decreasing.
Now performing a direct comparison between the cost of both controls we obtain
\begin{eqnarray*}
J(u^*)-J(\tilde{u})&=& \int\limits_0^T u^*(t)dt-\int\limits_0^T \tilde{u}(t)dt
= \int_{t_2}^{t_3}u^*(t)\, dt>0,
\end{eqnarray*}
which contradicts the optimality of $u^*$.
Therefore,  $\{u^*=M\}$ is a single interval and it follows that 
$$
\{u^*=0\}=\{p\in\{0,\theta\}\}=\{p=0\}\cup \{p=\theta\} .
$$
We conclude that all the mosquitoes are released in a single interval such that $|\{u^*=M\}|=T^*$, and thus the set $\{u^*=0\}$ splits into two intervals verifying $|\{u^*=0\}|=|\{p=0\}|+|\{p=\theta\}|=T-T^*$.

Using that for all $\xi\in[0,T-T^*]$, $u_{\xi}^*=M\mathbbm{1}_{(0+\xi,T^*+\xi)}$ satisfies this property, we get that $J(u_{\xi}^*)=MT^*$ and there are infinitely many solutions to Problem~\eqref{prob:Tfixed}.
\end{proof}

\section{Optimal control with a free horizon of time} 
In the previous section, we obtained infinitely many solutions for problem \eqref{prob:Tfixed}, showing that the system presents a natural time in which to act upon it, i.e. $T^*$. This is supported by the fact that we need to assume $T\geq T^*$ in order to have existence of solutions and  by the fact that in case $T>T^*$, then there exists one or two time interval of size $T-T^*>0$ in which $u^*(t)=0$. This motivates the introduction of a new problem, in which the final time $T$ is free. The functional we are interested in minimizing in this section is a convex combination of the cost $J$ used in the previous section, and the final time: 
\begin{equation}
\label{prob:Tfree}
\tag{$\mathcal{P}_{M}^{\alpha}$}
 \begin{cases}
\displaystyle   \inf_{\substack{u\in\mathcal{U}_{T,M}\\ T>0}} J_\alpha (T,u),\\
 p'=f(p)+ug(p), \ p(0)=0 \ , \ p(T)=\theta ,
 \end{cases}    
\end{equation}
where $\mathcal{U}_{T,M}$ is given by \eqref{def:UTM}, $\alpha\in[0,1]$ and
$$
J_\alpha (T,u)=(1-\alpha)\int_0^T u(s)ds + \alpha T .
$$
We expect to see the intervals where $u=0$ disappear. Let us state the main result of this section.

\begin{theorem}
\label{theo:Tfree}
Let $M>m^*$ and $\alpha>0$. Then, the unique solution of Problem~\eqref{prob:Tfree} is $u^*(t)=M\mathbbm{1}_{[0,T^*]}$ and $T^*$ given by \eqref{def:Tstarmstar}. Thus the optimal value is $J(u^*)=T^*\left((1-\alpha)M+\alpha\right)$.
\end{theorem}

\begin{proof}
Existence is investigated separately in Appendix \ref{app:existT}. Here, we focus on the characterization of the solution using first-order optimality conditions.
We begin by excluding the case $\alpha=1$, because the solution is trivial in this case. The problem reads
\begin{equation}
\tag{$\mathcal{P}_{T,M}^{1}$}
 \begin{cases}
  \inf\limits_{\substack{u\in\mathcal{U}_{T,C,M}\\ T>0}} T,\\
 p'=f(p)+ug(p), \ p(0)=0 \ , \ p(T)=\theta 
 \end{cases}    
\end{equation}
and the solution is clearly the constant function $u^*=M$. Indeed, it is easy to prove that, if the set $\{u^*<M\}$ has a positive measure, then one decreases the minimal time to reach $\theta$ by increasing $u^*$. On top, by the reasoning carried out in Theorem~\ref{theo:Tfixed}, we know that $T^*$ is the time that it takes for the system 
$$
\begin{cases}
&  p_M'=f(p_M)+Mg(p_M) ,\\
& p_M(0)=0 ,
\end{cases}
$$
to reach the point $p_M(T^*)=\theta$. The conclusion follows.

Now, let us assume that $\alpha<1$ and let $(T,u^*)$ denote an optimal pair. To apply the Pontryagin Maximum Principle (PMP), we need to define the Hamiltonian of the system
$$
\mathcal H: \R_+\times \R\times \R\times \{0,-1\}\times U\ni(t,p,q,q^0,u)=q(f(p)+ug(p)) + q^0 (1-\alpha) u .
$$
The equation that $q$ obeys reads
$$
 q'=-\frac{\partial \mathcal H}{\partial p}(t,p,q,q^0,u^*)=-q(f'(p)+u^*g'(p)) 
$$
and therefore, $q$ has a constant sign.

The instantaneous maximization condition reads
\begin{eqnarray*}
u^*(t)&\in&\arg\max\limits_{v \in U}\mathcal H(t,p,q,q^0,v)=\arg\max\limits_{v \in U}\left(w(t)+q^0(1-\alpha)\right)v,
\end{eqnarray*}
where $w(t)=q(t)g(p(t))$.

Thanks to the same reasoning done in Theorem~\ref{theo:Tfixed}, we get similarly that $q$ is positive.
Since the final time is free, we have the extra condition $u^*(T)=v_T$, where $v_T$ solves the one-dimensional optimization problem
$$
\max\limits_{ v_T \in U}\mathcal H(T,\theta,q(T),q^0,v_T)=\max\limits_{ v_T \in U} (w(T) + q^0(1-\alpha))v_T=-q^0\alpha .
$$
This condition rules out the case $q^0=0$ because, if we assume $q^0=0$, then we have $\max\limits_{0\leq v_T \leq M} w(T)v_T=0$, and since $w(t)>0$ then we must have at the same time $v_T=M$ and $w(T)M=0$, leading to a contradiction. Therefore $q^0=-1$, and we infer that the first order optimality conditions imply
\begin{equation}\label{metz:1000}
     \begin{cases}
      w(t)\leq 1-\alpha \mbox{ on } \{u^*=0\} ,\\
      w(t)=1-\alpha \mbox{ on } \{0<u^*<M\} , \\
      w(t)\geq 1-\alpha \mbox{ on } \{u^*=M\} , \\
      \max\limits_{0\leq v_T \leq M} (w(T) -(1-\alpha))v_T=\alpha .
     \end{cases}    
\end{equation}
Recall that $w$ is increasing if $p(t)\in(0,p^*)$  and decreasing if $p(t)\in(p^*,\theta)$, where $p^*$ is given by \eqref{def:petoile}.
Mimicking the reasoning done in proof of Theorem~\ref{theo:Tfixed} (in the case where $T$ is fixed), one can asserts that $u^*$ is bang-bang and that the set $\{u^*=M\}$ is one single open interval. Since $\alpha>0$, the last condition of \eqref{metz:1000} yields $u^*(T)>0$, so $u^*(T)=M$. This, together with the fact that the time it takes the system to reach $p=\theta$ at speed $M$ is $T^*$ allows us to conclude that solutions must be of the form $u^*(t)=M\mathbbm{1}_{[\xi,\xi+T^*]}$, with $\xi\geq 0$. Indeed, once the function has switched to $u^*=M$ it can not switch back to $u^*=0$.  Then, looking at the functional we want to minimize, we conclude that $\xi=0$, since the cost term is independent of $\xi$ and the term $\alpha T$ is increasing with respect to  $\xi$. As a result, the (unique) solution is $u^*=M\mathbbm{1}_{[0,T^*]}$.
\end{proof}

\noindent \textbf{Remark:} 
if we set $\alpha=0$ in problem \eqref{prob:Tfree}, we recover problem $\eqref{prob:Tfixed}$ but without a restriction on the final time. This allows $T$ going to infinity and explains that all the pairs of the form $(u_T,T)$ where
$$
T >T^*\quad \text{and}\quad u_T(t)=M\mathbbm{1}_{[\xi,\xi+T^*]},
$$
and $\xi\in [0,T-T^*)$, solve this problem. Once the system has reached the final state $p(T)=\theta$ it can stay there indefinitely without using any mosquitoes (i.e. $u=0$).
It is not so realistic from a practical point of view, and justifies that to fix $T$ to avoid the emergence of such noncompact families of solutions.

\section{Numerical simulations}

This section is devoted to some numerical simulations. We will use the Python package  GEKKO (see \cite{gekko}) which solves among other things optimal control problems under large-scale  differential equations thanks to nonlinear programming solvers.

This section is devoted to some numerical simulations. We will use the Python package  GEKKO (see \cite{gekko}) which solves among other things optimal control problems under large-scale  differential equations thanks to nonlinear programming solvers. 

A particular attention has been paid to the development of a user-friendly source code, which is available online at
\begin{center}
\url{https://github.com/jesusbellver/Minimal-cost-time-strategies-for-population-replacement-using-the-IIT}
\end{center}
with the hope that it serve as a useful basis for further investigations.

\subsection{1D Case}

Hereafter, we provide some simulations for the reduced problems \eqref{prob:Tfixed} and \eqref{prob:Tfree}, that can be seen as numerical confirmations of the theoretical results stated in Theorems~\ref{theo:Tfixed} and \ref{theo:Tfree}. The parameters considered for these simulations are given in Table~\ref{tab:1D}, according to the biological parameters considered in \cite{wolbachia}. 

\begin{table}
    \centering
    \caption{Simulation parameter values considered in \eqref{prob:Tfixed} and \eqref{prob:Tfree}}
    \begin{tabular}{|c|c|c|c|}
        Category & Parameter & Name & Value \\
        \hline 
        \multirow{2}{*}{Optimization} &  $T$ & Final time & 0.5 \\
        \cline{2-4}
        &  $M$ & Maximal release number & 10 \\
        \hline
        \multirow{5}{*}{Biology}&  $b_1^0$ & Normalized wild birth rate & 1 \\
        \cline{2-4}
         & $b_2^0$ & Normalized infected birth rate & 0.9 \\
         \cline{2-4}
         & $d_1^0$ & Wild death rate & 0.27 \\ 
         \cline{2-4}
         & $d_2^0$ & Infected death rate & 0.3 \\
         \cline{2-4}
         & $K$ & Normalized carrying capacity & 1 \\
         \cline{2-4}
         & $s_h$ & Cytoplasmatic incompatibility level & 0.9 \\
    \end{tabular}
    \label{tab:1D}
\end{table}

Simulations for Problem~\eqref{prob:Tfixed} are provided on Fig.~\ref{fig:1DTfixed}. To deal with the constraint $p(T) = \theta$, we added a penalization term in the definition of the functional.
As expected, we recover on Fig.~\ref{fig:1DTfixed} that the optimal control is bang-bang and that all mosquitoes are only released for an interval of time. Although other optimal solutions may exist, we obtain a particular one, where the action is concentrated at the beginning of the total time interval $[0,T]$.

\begin{figure}[h]
\centering
\includegraphics[width=\textwidth]{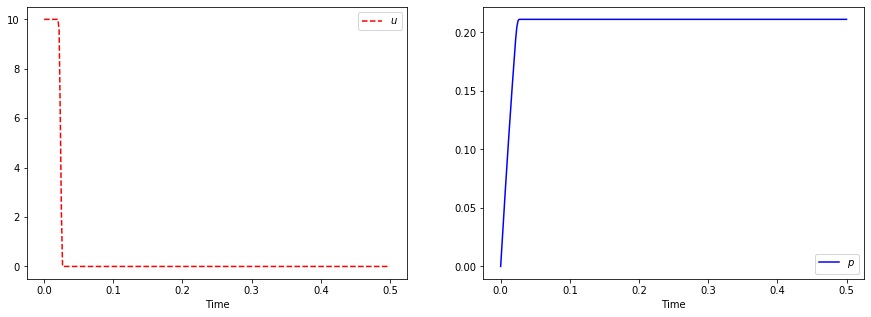}
    \caption{Simulation of the reduced problem \eqref{prob:Tfixed}. The penalization parameter is $\varepsilon=0.01$ and the number of elements in the ODEs discretization is equal to $300$.}
    \label{fig:1DTfixed}
\end{figure}

Simulations for Problem~\eqref{prob:Tfree} are provided on Fig.~\ref{fig:1DTfree}.
As proven in Theorem~\ref{theo:Tfree}, the effect of letting the final time $T$ free and adding it with a weight to the functional we want to minimize leads to the absence of an interval during which no action is taken, for every $\alpha\in (0,1]$. On Fig.~\ref{fig:1DTfree}, numerical solutions are plotted for $\alpha=0.01$. The final time obtained is $T^*_{\rm num}\approx 0.0238519$, which is very closed to the expected theoretical one 
$$
T^*=\int_0^{\theta}\frac{d\nu}{f(\nu)+Mg(\nu)}\approx 0.0238122.
$$
Results for other values of $\alpha$ are similar.

\begin{figure}[h]
\centering
\includegraphics[width=\textwidth]{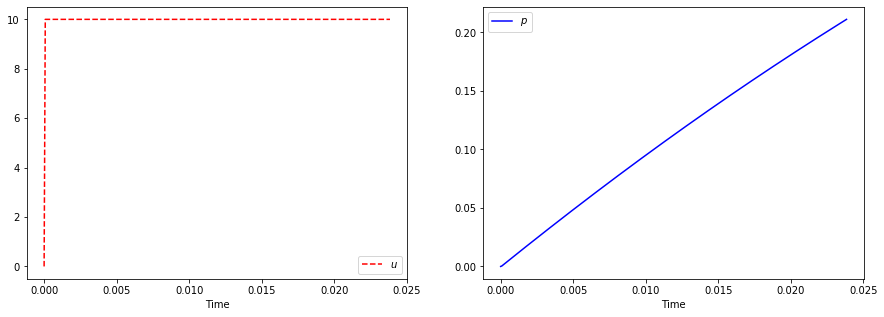}
    \caption{Simulation of the reduced problem \eqref{prob:Tfree} with $\alpha=0.01$. The number of elements in the ODEs discretization is equal to $300$.}
    \label{fig:1DTfree}
\end{figure}

\subsection{2D Case}

In this section, we provide simulations for optimal control problems involving the full system \eqref{eq:fullsys}. We will use the parameters of Table~\ref{tab:2D}, where the biological parameters have been extracted from \cite{bossin}. We choose $M$ to be ten times higher than the birth rate of wild mosquitoes, in analogy with the simulations for the reduced problem.

\begin{table}
    \centering
    \caption{Simulation parameters value for the optimal control problems involving the full system \eqref{eq:fullsys}}
    \begin{tabular}{|c|c|c|c|}
        Category & Parameter & Name & Value \\
        \hline
        Optimization & $M$ & Maximal release number & 112\\
        \hline
        \multirow{5}{*}{Biology}&  $b_1$ & Wild birth rate & 11.2 \\
        \cline{2-4}
         & $b_2$ & Infected birth rate & 10.1 \\
         \cline{2-4}
         & $d_1$ & Wild death rate & 0.04 \\ 
         \cline{2-4}
         & $d_2$ & Infected death rate & 0.044 \\
         \cline{2-4}
         & $K$ & Carrying capacity & 5124 \\
         \cline{2-4}
         & $s_h$ & Cytoplasmatic incompatibility level & 0.9 \\
    \end{tabular}
    \label{tab:2D}
\end{table}

To compute the carrying capacity $K$, we used the same procedure as in \cite{bossin}, but adapting it to our model. We will make our results relevant for an island of $74$~ha with a mosquito density of $69$~ha$^{-1}$, so the amount of wild mosquitoes at the equilibrium is $n_1^*=74\times 69 = 5106$. Then, since $n_1^*=K\left(1-\frac{d_1}{b_1}\right)$, we obtain the following carrying capacity of the environment 
$$
K=\frac{n_1^*}{1-\frac{d_1}{b_1}}\approx 5124.3011.
$$

We first deal with the case where $T$ is fixed, in other words, we solve the optimal control problem 
\begin{equation}\label{full:CPbis}
\inf\limits_{u\in\mathcal{U}_{T,M}} \int_{0}^{T}u(s)\, ds+\frac{1}{\varepsilon}
\max\left\{n_1(T)-10,n_2^*-10-n_2(T),0\right\},
\end{equation}
where $\Vert\cdot \Vert_{\R^2}$ stands for the Euclidean norm in $\R^2$.
The minimized criterion is a combination of the total amount of mosquitoes used and a penalization term standing for the final distance to 
to the region $[0,10]\times[n_2^*-10,n_2^*]$, with $\varepsilon=0.0001$. 

It is easy to show that, because of the pointwise constraints on the control function $u(\cdot)$, the steady-state $(0,n_2^*)$ of System~\eqref{eq:fullsys} cannot be reached in time $T$. This is why we chose to penalize the final distance to an arbitrary region that is clearly included in the basin of attraction of the steady-state $(0,n_2^*)$ but is reachable.

The simulations are performed for different final times, and results are given on Fig.~\ref{fig:2DTfixed}.

\begin{figure}[h!]
    \centering
    \includegraphics[width=\textwidth]{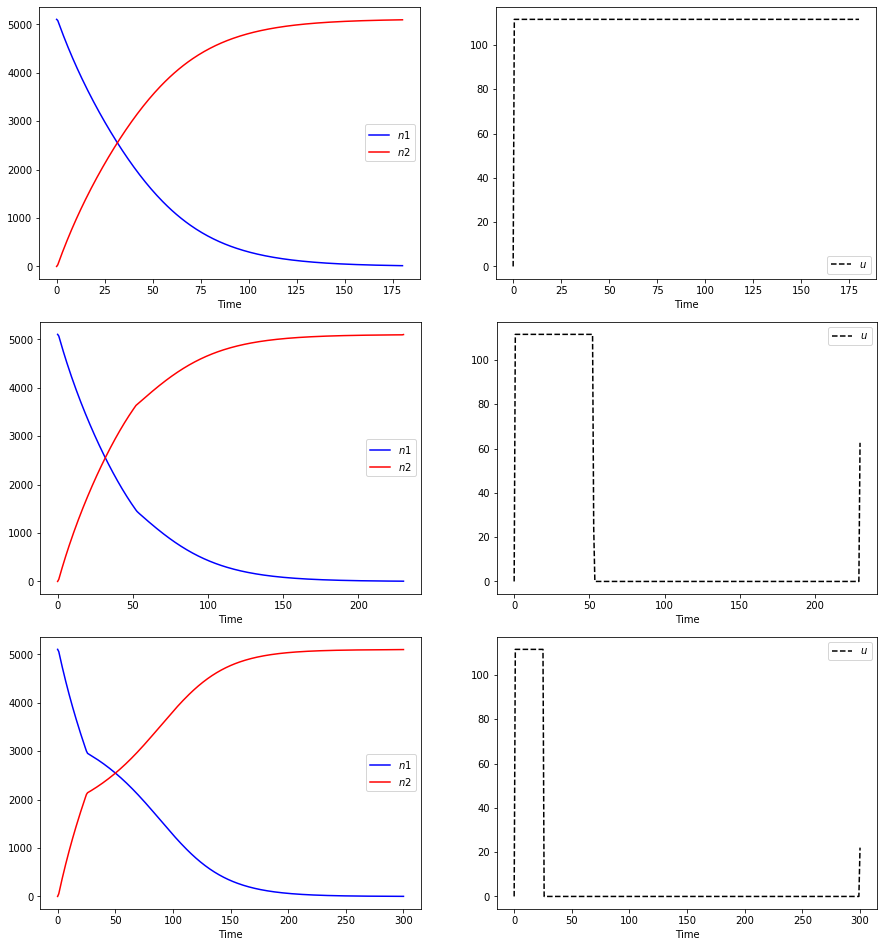}
    \caption{Simulation of the full problem~\eqref{full:CPbis} with $T$ fixed for $T=195$ (first row), $T=210$ (second row) and $T=250$ (third row). The time step in the ODEs discretization  is $\Delta t=T/300$.}
    \label{fig:2DTfixed}
\end{figure}

On Fig.~\ref{fig:2DTfixed}, we observe that if $T$ is not large enough to get close enough to the point $(0,n_2^*)$, then the control is the function $u$ equal to $M$ almost everywhere. When $T$ increases, the action is carried out in two stages: first, one has $u=M$ at least until the system enters the basin of attraction of the equilibrium point $(0,n_2^*)$, then $u=0$ to let the system evolve without using mosquitoes. The large $T$ is, the less it seems necessary to act. 
A possible explanation is that with only a little action, it is possible to enter the basin of attraction of $(0,n_2^*)$. Therefore, if $T$ is big, we can stop acting soon to decrease the amount of mosquitoes used. Instead, if $T$ is small, we need to release a lot of mosquitoes because otherwise, the system would not get close enough to $(0,n_2^*)$ alone.

\medskip

Finally, on Fig.~\ref{fig:2DTfree}, simulations are carried out for the full system \eqref{eq:fullsys}, by letting $T$ free and replacing the cost previous functional by
$$
u\mapsto (1-\alpha)\int_0^Tu(t)dt +\alpha T+\frac{1}{\varepsilon}
\max\left\{n_1(T)-10,n_2^*-10-n_2(T),0\right\},
$$ 
with $\alpha \in [0,1]$. 
\begin{figure}[h!]
    \centering
    \includegraphics[width=\textwidth]{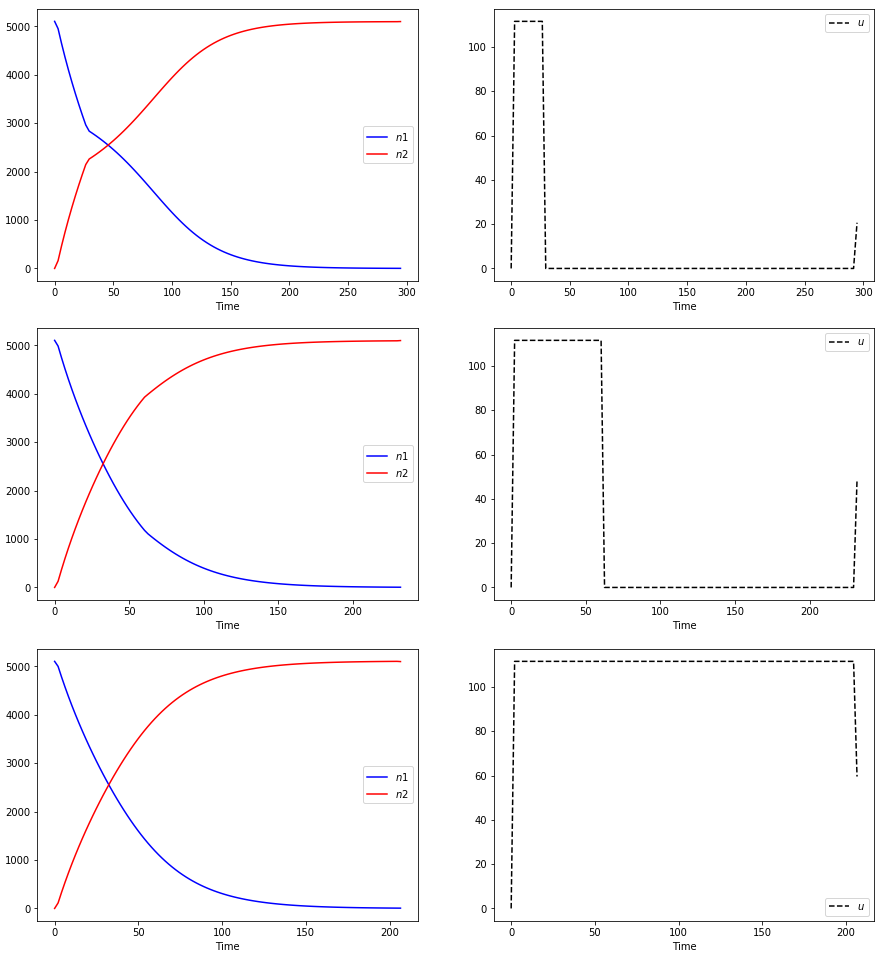}
    \caption{Simulation of the full problem with $T$ free for $\alpha=0.1$ (first row), $\alpha=0.5$ (second row) and $\alpha=0.9$ (third row).The number of points in the ODEs discretization  is $101$. }
    \label{fig:2DTfree}
\end{figure}
We see that the effect of increasing $\alpha$, and thus giving more importance in the horizon of time $T$, has the the same effect as decreasing $T$ in the case with $T$ fixed. In these simulations, the final times obtained are: $T=245.4$ for $\alpha=0.1$, $T=206.4$ for $\alpha=0.5$ and $T=190.9$ for $\alpha=0.9$. The results obtained are very similar to the the ones with $T$ fixed and very close to the duration during which the control is acting.

\appendix


\section{Existence of solutions for the problem \eqref{prob:Tfixed}}\label{app:existT}

The set of admissible controls for Problem~\eqref{prob:Tfixed} is
$$
\mathcal{D}=\left\{u\in\mathcal{U}_{T,M} \ , \ p(T)\geq\theta\right\}.
$$

Let us first prove that $\mathcal{D}$ is non-empty. This leads to investigate the assumptions under which one can ensure that constant controls $u(t)=\Bar{u}\mathbbm{1}_{[0,T]}$ belong to the set $\mathcal{D}$. 
Let us introduce $p_{\bar{u}}$ solving
$$
\begin{cases}
&  p_{\bar{u}}'=f(p_{\bar{u}})+\bar{u}g(p_{\bar{u}})\quad \text{in (0,T)},\\
& p_{\bar{u}}(0)=0.
\end{cases}
$$
By integrating in both sides of the differential equation, we get that the time it takes for $p_{\bar{u}}$ to reach the point $\theta$, called $T_{\bar{u}}$, is
$$
 T_{\bar{u}}= \int_{0}^{\theta}\frac{d\nu}{f(\nu)+\bar{u}g(\nu)} .
 $$

Note that $T_{\bar{u}}$ is finite if one imposes $\bar{u}>m^*$. Also by the fact that we want $\bar{u}\in\mathcal{U}_{T,M}$, we have that $\bar{u}\leq M$, so $\bar{u}\in]m^*,M]$. Finally since the final time $T$ is fixed, we need $T_{\bar{u}}\leq T$, and using that $\int_{0}^{\theta}\frac{d\nu}{f(\nu)+\bar{u}g(\nu)}$ is decreasing with respect to $\bar{u}$, we deduce that $T^*\leq T_{\bar{u}}$. Thus, we can conclude that $\mathcal{D}$ contains at least one constant control, and therefore is non-empty if, and only if, $T^*\leq T$, as we assumed.

Using now the fact that  $\mathcal{D}$ is non-empty, we consider a minimizing sequence $\left(u_n\right)_{n\in\mathbb{N}}\in\mathcal{D}^{\mathbb{N}}$ for the problem \eqref{prob:Tfixed}. We have $0\leq u_n \leq M$ a.e. for all $n\in\mathbb{N}$. Hence, the sequence $\left(u_n\right)_{n\in\mathbb{N}}$ is uniformly bounded. Also, since $\left(L^1(0,T)\right)'=L^{\infty}(0,T)$ and using the Banach-Alaouglu theorem, we conclude that $\mathcal{D}$ is weakly-* compact and therefore, up to a subsequence, $u_{n}\underset{n\to\infty}{\rightharpoonup^*} u^*$, i.e. $\left(u_n\right)_{n\in\mathbb{N}}$ converges for the weak-* topology of $L^{\infty}(0,T)$, and $0\leq u^*\leq M$ a.e., so that $u^*\in\mathcal{D}$. 

We now consider $\left(p_n\right)_{n\in\mathbb{N}}$ where $p_n$ solves  $p'_n=f(p_n)+u_ng(p_n)$ with $p_n(0)=0$. Using the fact that $f,g\in\mathcal{C}^{\infty}([0,1])$ 
and $0\leq p_n\leq 1$, we deduce that $(p'_n)_{n\in \N}$ is bounded in $L^{\infty}(0,T)$. 
Hence $p_n\in\mathcal{C}^{0}([0,1])$ and therefore, by using the Ascoli-Arzelá theorem, we conclude that up to a subsequence, $p_n \overset{\mathcal{C}^{0}}{\longrightarrow} p^*$ where $p^*\in W^{1,\infty}(0,T)$.

To conclude, 
since $u_{n}\underset{n\to\infty}{\rightharpoonup^*} u^*$, then  $\int_{0}^{T}\varphi u_n\to\int_{0}^{T}\varphi u^*$ for all $\varphi\in L^1(0,T)$.  In particular, 
for $\varphi:t\mapsto 1$, we have $\int_{0}^{T}1\cdot u_n\to\int_{0}^{T}1\cdot u^*$. Hence $J(u^*)=\lim\limits_{n\to\infty} J(u_n)=\inf\limits_{u\in\mathcal{D}} J(u)$, and therefore problem \eqref{prob:Tfixed} admits a solution.


\section{Existence of solutions for the problem \eqref{prob:Tfree}}

In order to simplify the study of the existence of solutions, and to avoid working on a variable domain, we make the following changes to variables. We define $\tilde{p}(s):=p(Ts)$ and $\tilde{u}(s):=u(Ts)$, $s\in[0,1]$.

Then, we are led to consider the problem
\begin{equation*}
\label{prob:varP}
\tag{$\tilde{\mathcal{P}}^{\alpha}_{M}$}
    \begin{cases}
    & \tilde{p}'(s)=T\left(f(\tilde{p}(s))+\tilde{u}(s)g(\tilde{p}(s))\right) \mbox{ , } \tilde{p}(0)=0 \mbox{ , } \tilde{p}(1)=\theta, \\
    & \displaystyle \inf_{\tilde u\in L^\infty(0,1;[0,M])}\tilde{J}(T,\tilde{u}),
    \end{cases}
\end{equation*}
where $\tilde{J}(T,\tilde{u})$ is defined by
$$
\tilde{J}(T,\tilde{u})=(1-\alpha) \ T\int_{0}^{1}\tilde{u}(s)ds + \alpha T .
$$
Since the solutions of this system are the same as those of the system we are interested in, we will study the existence of solutions for the new one.

Let us define the set of $(T,\tilde{u})$ verifying the constraints of System~\eqref{prob:varP}, i.e.
$$
\mathcal{D}:=\left\{\left(T,\tilde{u}\right)\in\mathbb{R}^+\times\mathcal{U}_{1,M}\times[0,1] \ | \ \tilde{p}(1)\geq \theta \right\}.
$$
This set is clearly non-empty (consider for instance $T=T^*$ and $u(\cdot)=M\mathbbm{1}_{[0,T^*]}(T^*\cdot)$). Consider a minimizing sequence $\left(T_n,\tilde{u}_n\right)_{n\in\mathbb{N}}\in\mathcal{D}^{\mathbb{N}}$ and let $\tilde p_n$ be associated solution of the ODE definig Problem~\eqref{prob:varP}. 

By minimality, one has $\lim\limits_{n\to\infty}\tilde{J}(T_n,\tilde{u}_n)<\infty$, i.e.
$$\lim\limits_{n\to\infty} \ (1-\alpha) \ T_n\int_{0}^{1}\tilde{u}_n(s)ds+\alpha T_n <\infty.$$
Each term of the sum being bounded from below by 0 is also bounded. Since $\alpha>0$, $(T_n)_{n\in \N}$ is bounded, therefore, up to a subsequence, $T_n\to \tilde T<\infty$. By mimicking the arguments used in Section~\ref{app:existT}, one shows that up to a subsequence, $(\tilde{u}_n)_{n\in \N}$ converges to $\tilde{u}^*\in \mathcal{U}_{\tilde T,M}$ weakly-star in $L^\infty(0,1;[0,M])$ and $(\tilde{p}_n)_{n\in \N}$ converges to $\tilde{p}^*$ in $\mathcal{C}^0([0,\tilde T])$, where $\tilde p^*$ solves the equation
$$
(\tilde p^*)'=f(\tilde p^*)+\tilde u^*g(\tilde p^*), \quad \text{in }(0,\tilde T)
$$
and $\tilde p^*(0)=0$. As a consequence, $(\tilde{J}(T_n,\tilde{u}_n))_{n\in \N}$ converges to $\tilde{J}(\tilde T,\tilde{u}^*)$, which concludes the proof.

\section*{Acknowledgement}
the authors are supported by the Project  ``Analysis and simulation of optimal shapes - application
to life science'' of the Paris City Hall.

\bibliographystyle{abbrv}
\bibliography{article}


\end{document}